\documentclass[oneside,reqno,12pt]{amsart}
\usepackage[greek,english]{babel}
\usepackage{amsthm}
\usepackage{amsbsy}
\usepackage{amsfonts}
\usepackage{graphicx}
\usepackage{hyperref}
\hypersetup{colorlinks=true, citecolor=blue, linkcolor=red}

\newtheorem{thm}{Theorem}
\newtheorem{cor}{Corollary}

\newtheorem{rem}{Remark}

\numberwithin{equation}{section} \numberwithin{lem}{section}
\numberwithin{thm}{section} \numberwithin{cor}{section}
\numberwithin{pro}{section} \numberwithin{rem}{section}

\begin{document}
\title[Painlev\'{e}-II profile of the shadow kink]{Painlev\'{e}-II profile of the shadow kink in the theory of light-matter interaction in nematic liquid crystals}
\author{Christos Sourdis}
\address{Institute of Applied and Computational Mathematics, FORTH, GR–711
10 Heraklion, Greece.
}
\email{sourdis@uoc.gr}

\date{\today}
\begin{abstract}
 We confirm a prediction   that the recently introduced  shadow kink defect in the theory of light-matter interaction in nematic liquid crystals is described to main order by a  solution  of the Painlev\'{e}-II equation which changes sign once in the whole real line. Our result implies  the existence of such a solution to the latter equation which is energy minimizing with respect to compactly supported perturbations.
 \end{abstract}
 \maketitle

\section{Introduction}



Motivated by experiments in  light-matter interaction in nematic liquid crystals, the authors of \cite{clerc2017theory}
considered energy minimizing in $H^1(\mathbb{R})$ solutions of the following singular perturbation problem:
\begin{equation}\label{eqEq}
  \epsilon^2 v''(x)+\mu(x)v(x)-v^3(x)+\epsilon a f(x)=0,\ \ x\in \mathbb{R},
\end{equation}
where $a>0$ is a parameter, and the fixed functions  $\mu, f$ satisfy
\begin{equation}\label{eqAssump}
  \left\{\begin{array}{c}
           \mu \in C^1(\mathbb{R})\cap L^\infty(\mathbb{R})\ \textrm{is\ even}, \ \mu'<0\ \textrm{in}\ (0,\infty),\\
   \mu(\xi)=0\ \textrm{for some}\ \xi>0;          \\ \\
           f\in L^1(\mathbb{R})\cap L^\infty(\mathbb{R})\cap C(\mathbb{R})\ \textrm{is\ odd}, \ f>0\ \textrm{in}\ (0,\infty).
         \end{array}
   \right.
\end{equation}

It was shown that, for any $\epsilon, a>0$,    minimizers $v_\epsilon$ always exist and have a unique zero $\rho_\epsilon$. The main result of \cite{clerc2017theory} can be summarized as follows.   There exist positive numbers $a_*\leq a^*$  (independent of $\epsilon$) that can be characterized variationally as
\begin{equation}\label{eqastar}\begin{array}{l}
                                 a^*=\sup_{x\in [-\xi,0)}\frac{\sqrt{2}\left((\mu(0))^\frac{3}{2}-(\mu(x))^\frac{3}{2} \right)}{3\int_{x}^{0}|f|\mu}<\infty, \\
                                  \\
                                 a_*=\inf_{x\in (-\xi,0]}\frac{\sqrt{2}(\mu(x))^\frac{3}{2} }{3\int_{-\xi}^{x}|f|\mu}\in(0,\infty),
                               \end{array}
\end{equation}
 such that:
\begin{itemize}
\item If $a\in (a^*,\infty)$, then the minimizers $v_\epsilon$ converge pointwise to
$\textrm{sgn}(x)\sqrt{\mu^+(x)}$ as $\epsilon \to 0$ (the convergence being uniform away from the origin, and so $\rho_\epsilon\to 0$). There is a transition layer around the origin of width $O(\epsilon)$, where the behaviour of $v_\epsilon$ is governed by  a usual squeezed hyperbolic tangent profile (see also (\ref{eqHyper}) below).
  \item  If $a\in (0,a_*)$, then  $v_\epsilon$ converges uniformly to
$\sqrt{\mu^+(x)}$  and  $\rho_\epsilon \to -\xi$ as $\epsilon \to 0$ (without loss of generality).
\end{itemize}

In any case, $v_\epsilon$ has steep corner layers of width $O(\epsilon^\frac{2}{3})$ around $\pm \xi$, where the behaviour of $v_\epsilon$ is governed by a suitable blow-down of a minimizing (in the sense of Morse) solution of the Painlev\'{e}-II equation:
\begin{equation}\label{eqPainle}
  y''(s)-sy(s)-2y^3(s)-\alpha=0,\ \ s\in \mathbb{R},
\end{equation}
with
\begin{equation}\label{eqAlpha}
  \alpha=\frac{af(-\xi)}{\sqrt{2}\mu'(-\xi)}<0,
\end{equation}
after appropriate reflections and rescalings of constants. More precisely, assuming that $v_\epsilon\to \sqrt{\mu^+}$ (without loss of generality), it holds
\begin{equation}\label{eqSatur}
  2^{-\frac{1}{2}}\left(\mu'(-\xi)\epsilon \right)^{-\frac{1}{3}}v_\epsilon\left(-\xi-\frac{\epsilon^\frac{2}{3}s}{\left(\mu'(-\xi)\right)^\frac{1}{3}} \right)\to -Y(s)\ \textrm{in}\ C^1_{loc}(\mathbb{R})\ \textrm{as}\ \epsilon\to 0,
\end{equation}
where $Y$ is an energy minimizing solution of (\ref{eqPainle}) with $\alpha$ as in (\ref{eqAlpha}). We clarify that here minimality is understood in the sense of Morse or De Giorgi, that is with respect to compact perturbations. In fact,  such solutions automatically satisfy either the plus or the minus set of asymptotic conditions in (\ref{eqY+}) below (see again \cite{clerc2017theory}).

It is known that for any $\alpha\leq 0$ the equation (\ref{eqPainle}) admits a unique positive solution $Y_+$ which is defined in the whole real line (see \cite{hastings1980boundary} for $\alpha=0$, and \cite{clerc2017theory,fokas2006painleve} for $\alpha<0$). In fact, it holds $Y_+'<0$, which implies that $Y_+$ is linearly nondegenerate, in the sense that the linearization of (\ref{eqPainle}) about $Y_+$ does not have bounded elements in its kernel.
 On the other hand, for $\alpha<0$ in  some    neighborhood of zero, the problem (\ref{eqPainle}) has also a  solution $Y_-$ which changes sign once \cite{troy2018role} (we refer to the discussion following Corollary \ref{cor3} for more details and references).
 We point out that it holds
\begin{equation}\label{eqY+}
  Y_\pm(s)\sim \pm \sqrt{-\frac{s}{2}}+O\left(\frac{1}{s}\right)\ \textrm{as}\ s\to -\infty;\ \ Y_\pm(s)\sim \frac{|\alpha|}{s}\ \textrm{as}\ s\to +\infty.
\end{equation}

A natural question that arises  is   whether the minimizing solution  $Y$ of (\ref{eqPainle}) that appears in (\ref{eqSatur}) is $Y_+$ or   a solution that changes sign exactly once (recall that $v_\epsilon$ changes sign exactly once but its zero $\rho$ could escape at infinity in the blow-up (\ref{eqSatur})). 
In this paper we will show that it is indeed the second, and more interesting, scenario that occurs. Thus, we describe the microstructure of $v_\epsilon$ near $-\xi$ (at least in an inner zone of width $O(\epsilon^\frac{2}{3})$). We confirm the prediction that the microscopic phase transition that takes place near $-\xi$ is   a new type  of defect that does not involve the standard hyperbolic tangent. The fact that this is an open problem was emphasized at the end of the introduction in \cite{clerc2018gradient}.

Our main result is the following.
\begin{thm}\label{thm1}
If in addition to (\ref{eqAssump}) we assume that $\mu,\ f\in C^2$ near $-\xi$, then
the minimizing solution  $Y$ of (\ref{eqPainle}) that appears in (\ref{eqSatur})   changes sign exactly once, provided that $a\in (0,a_*)$ with $a_*$ as in (\ref{eqastar}).
\end{thm}

Our result  implies   the following.

\begin{cor}\label{cor3}
  If $\alpha\in  (-\frac{1}{2},0)$,  the   problem (\ref{eqPainle}) admits  a   solution $Y_-$ with exactly one sign change that is energy minimizing with respect to compactly supported perturbations. Moreover, the solution $Y_-$ satisfies the corresponding conditions in (\ref{eqY+}).
\end{cor}

Let us compare our above result with some related recent ones in the literature. In \cite{claeys2008multi,fokas2006painleve}, by solving the corresponding Riemann-Hilbert
problem, it was shown that there exists a unique solution of (\ref{eqPainle})-(\ref{eqY+})$_-$ for $\alpha \in (-\frac{1}{2},0)$. Moreover, it was shown in the latter reference that this solution
is a meromorphic function with no poles on the real line. More recently, the existence of a solution to (\ref{eqPainle})-(\ref{eqY+})$_-$ with exactly one sign change
 was shown by a shooting argument in \cite{troy2018role} for $\alpha \in (\alpha_*,0)$, where $ \alpha_*<0$ is some    number that is sufficiently close to zero. The approach in the latter reference provides a qualitative description of the obtained solution.  In particular, it changes monotonicity
only once. It is worth mentioning that the existence and uniqueness of such a nonmonotone  solution  for $\alpha\in (-\frac{1}{2},0)$ were indicated in
\cite{fornberg2014computational}. Interestingly enough, it was also indicated therein that there are no such solutions for $\alpha=-\frac{1}{2}$.

In relation to the aforementioned papers, where different approaches are employed, our main contribution  is that we associated to the   solution  of (\ref{eqPainle})-(\ref{eqY+})
the strong property of energy minimality. In general, in problems of the calculus of variations, the standard method for showing this property is by constructing a calibration. However,  it is not clear to us how to apply this considerably  more direct approach to the problem at hand. Nevertheless, we note in passing that this is indeed possible for  $Y_+$, based on the fact that it is sign definite. Actually, it follows from our proof that the validity of Corollary \ref{cor3} holds for $\alpha\in (\alpha_*,0)$, where $\alpha_*$ is as in the left hand side of (\ref{eqastarInf}). In this regard, we point out that it was mentioned at the end of the introduction of \cite{clerc2017theory} that numerical evidence suggests that the range of
$\alpha$ in Corollary \ref{cor3} should be larger. Lastly, we should mention that by the B\"{a}cklund transformation (see for instance \cite{fokas2006painleve})
our result implies the existence of a solution to (\ref{eqPainle})-(\ref{eqY+})$_-$ for $\alpha \in (-k-\frac{1}{2},-k)$, $k\in \mathbb{N}$.

Our proof of Theorem \ref{thm1} goes by contradiction. Assuming that the assertion  is false puts us sufficiently far from $-\xi$, in some sense, which allows us to    compare $v_\varepsilon$ to   a sign definite solution $\eta_\epsilon$ of (\ref{eqEq}) by  examining
the quotient \[w_\epsilon=\frac{v_\epsilon}{\eta_\epsilon}.\]  This division trick was introduced by \cite{lassoued1999ginzburg} in a different context and has been used extensively in the study of
vortices in Bose-Einstein condensates (see \cite{ANS} and the references therein); the simplest case being (\ref{eqEq}) with $a=0$ and $-\mu$ being a trapping potential. In this regard, we point out that the connection between (\ref{eqEq}) and the aforementioned studies was already discussed in \cite{clerc2017theory}. However, to the best of our knowledge, this important division trick is applied here for the first time in this context. Thankfully, the required estimates for $\eta$ and its derivatives are readily derivable from \cite{karali2015ground} or   \cite{schecter2010heteroclinic}. These rely on the fact that the positive solution $Y_+$ of (\ref{eqPainle}) is linearly nondegenerate. The   quotient $w$ satisfies a weighted Allen-Cahn equation with a weight that degenerates at $-\xi$ (see (\ref{eqWstr}) and the discussion leading to (\ref{eqEepsil})).
However, as we remarked, we will be working
sufficiently away from this degeneracy. It then  boils down to
showing that $w$  cannot have a sharp transition layer from $1$ to $-1$ at $\rho$.
For this purpose, there are several approaches in the literature for related problems of Allen-Cahn type. For example see
  \cite[Thm. 1.5]{ambrosetti2003singularly} for a nonlinear Schr\"{o}dinger equation with a potential,
 \cite[Thm 4.1]{alikakos1988singular} for a phase field model of phase transitions, and  \cite[Thm. 1]{nakashima2003multi} for a spatially inhomogeneous Allen-Cahn equation.
On the one hand, the equation for $w$ resembles more the aforementioned phase field model. On the other hand, armed with our estimates for $w$,
we found it more convenient to adapt the corresponding proof of \cite{ambrosetti2003singularly} with our own twist.


We close this introduction by expressing our hope that some of our arguments can be extended to describe   the core of  the 'shadow vortex' defect in the recent paper \cite{clerc2018symmetry} or the 'shadow domain wall' in \cite{clerc2018gradient}.

In the rest of the paper we will prove our main results. Following the proof of Theorem \ref{thm2}, which implies Theorem \ref{thm1},  we will present in Remark \ref{remFormal} an applied mathematicians point of view which partly motivated our rigorous analysis.

\section{Proof of the main results}\subsection{Proof of Theorem \ref{thm1}}

In order to prove Theorem \ref{thm1}, as was already observed at the end of Section 3 of \cite{clerc2017theory}, it suffices to establish the following (recall also the preamble to the aforementioned theorem).

\begin{thm}\label{thm2} Let $\mu,f$ satisfy (\ref{eqAssump}) with $\mu,\ f\in C^2$ near $-\xi$.
  Then, if   $a$ is as in Theorem \ref{thm1}, the unique root $\rho_\epsilon$ of    $v_\epsilon$  satisfies
  \[
  \rho_\epsilon+\xi=O(\epsilon^\frac{2}{3})\ \ \textrm{as}\ \epsilon\to 0.
  \]
\end{thm}

\begin{proof}
Throughout the proof, we will denote by $C/c>0$ a large/small generic constant that is independent of small $\epsilon>0$, and whose value will increase/decrease as we proceed.

Firstly, we
will show that
\begin{equation}\label{eqBaggio}
  \rho_\epsilon\geq -\xi-C\epsilon^\frac{2}{3}.
\end{equation}
To this end, let us consider the algebraic equation that comes from (\ref{eqEq}) when we ignore the term $\epsilon^2v''$:
\begin{equation}\label{eqnueps}
\mu(x)\nu-\nu^3 +\epsilon a f(x)=0.
\end{equation}
Based on the fact that $\mu'(-\xi)>0$, it is straightforward to verify that there exists a large $R>0$ such that the following properties hold. The above  equation admits a unique solution $\nu_\epsilon$ for $x\leq -\xi-R\epsilon^\frac{2}{3}$, provided that $\epsilon>0$ is sufficiently small.
In fact, it holds
\begin{equation}\label{eqLast0}
  \nu_\epsilon(x)=-\epsilon a \frac{f(x)}{\mu(x)}+O\left(\frac{\epsilon^3}{|x+\xi|^4} \right)<0,
\end{equation}
uniformly on $[-\xi-d,-\xi-R\epsilon^\frac{2}{3}]$ for some small $d>0$ (such that $f\in C^2(-\xi-2d,-\xi+2d)$), as $\epsilon \to 0$, and
\[
\left|\nu_\epsilon'(x) \right|\leq C\frac{\epsilon}{|x+\xi|^2},\ \   \left|\nu_\epsilon''(x) \right|\leq C\frac{\epsilon}{|x+\xi|^3},\ \ x\in[-\xi-d,-\xi-R\epsilon^\frac{2}{3}].
\]
Let
\begin{equation}\label{eqLastDef}
  \varphi_\epsilon(x)=v_\epsilon(x)-\nu_\epsilon(x),\ \ x\in[-\xi-d,-\xi-R\epsilon^\frac{2}{3}].
\end{equation}
It follows readily that $\varphi_\epsilon$ satisfies the following linear equation:
\begin{equation}\label{eqLast1}
-\epsilon^2\varphi''+Q(x)\varphi=O\left(\frac{\epsilon^3}{|x+\xi|^3} \right),
\end{equation}
with
\begin{equation}\label{eqLast2}\begin{array}{ccc}
                                 Q(x) & := & \frac{v_\epsilon^3(x)-\mu(x)v_\epsilon(x)-a\epsilon f(x)}{v_\epsilon(x)-\nu_\epsilon(x)} \\
                                   &   &   \\
                                   & =  &  \varphi^2+3\nu_\epsilon^2+3\nu_\epsilon \varphi-\mu \\
                                  &   &   \\
                                   & \geq & -\mu \\
                                   & &\\
                                   &\geq&c|x+\xi|,
                               \end{array}
\end{equation}
for $x\in[-\xi-d,-\xi-R\epsilon^\frac{2}{3}]$.

It follows from Lemma 3.1 in \cite{clerc2017theory} that
\begin{equation}\label{eqLast3}\left|v_\epsilon(-\xi-R\epsilon^\frac{2}{3})\right|\leq C\epsilon^\frac{1}{3}.\end{equation}
On the other side, it follows as in Proposition 2.1 in \cite{ignat2006critical} that
\begin{equation}\label{eqLast4}
\varphi_\epsilon(-\xi-d)=O(\epsilon^2)\ \  \textrm{as}\ \ \epsilon\to 0.
\end{equation}
In light of (\ref{eqLast1}), (\ref{eqLast2}), (\ref{eqLast3}), (\ref{eqLast4}), we deduce by a standard barrier argument that
\[
\left|\varphi_\epsilon(x)\right|\leq C\epsilon^\frac{1}{3}e^{c\left\{\frac{x+\xi}{\epsilon^\frac{2}{3}}\right\}}+C\epsilon^2+C\frac{\epsilon^3}{|x+\xi|^4},\ \ x\in[-\xi-d,-\xi-R\epsilon^\frac{2}{3}].
\]
By combining (\ref{eqLast0}), (\ref{eqLastDef}) and the above relation, we arrive at
\[
v_\epsilon(x)=-\epsilon a \frac{f(x)}{\mu(x)}\left[1+O\left(\frac{\epsilon^2}{|x+\xi|^3} \right)+O\left(\frac{|x+\xi|}{\epsilon^\frac{2}{3}}e^{c\left\{\frac{x+\xi}{\epsilon^\frac{2}{3}}\right\}}\right)+O\left(\epsilon|x+\xi|\right)\right],
\]
uniformly on $[-\xi-d,-\xi-R\epsilon^\frac{2}{3}]$, as $\epsilon\to 0$.
Assuming that (\ref{eqBaggio}) does not hold, then we can evaluate  the above relation  at $x=\rho_\epsilon$ along a contradicting sequence of $\epsilon$'s that tend to zero such that
\[
\frac{\rho_\epsilon+\xi}{\epsilon^\frac{2}{3}}\to -\infty.
\]
This gives us that
\[
 1+O\left(\frac{\epsilon^2}{|\rho_\epsilon+\xi|^3} \right)+O\left(e^{c\left\{\frac{\rho_\epsilon+\xi}{\epsilon^\frac{2}{3}}\right\}}\right)+o\left(\epsilon\right) =0,
\]
which is clearly  a contradiction. We have thus established the desired lower bound in (\ref{eqBaggio}).

The main effort will  be to establish the 'opposite direction' of (\ref{eqBaggio}). To this end, we will argue by contradiction. Assuming  that the assertion of the theorem is false would give  us a sequence   $\epsilon_n\to 0$ such that
\begin{equation}\label{eqcontra}
    \frac{\rho_{\epsilon_n}+\xi}{\epsilon_n^\frac{2}{3}}\to +\infty.
\end{equation}

Abusing notation, we will drop from now on the subscript $n$ and assume that all the following $\epsilon's$ are along this sequence or a subsequence of it.

Let $\eta_\epsilon$ be the unique negative solution of
\[\left\{\begin{array}{c}
                                     \epsilon^2 \eta''(x)+\mu(x)\eta(x)-\eta^3(x)+\epsilon a f=0,\  x<0; \\
                                     \eta(-\infty)=0,\  \eta'(0)=0.
                                   \end{array} \right.
\]
In passing, we note that uniqueness follows as in Remark 9 in \cite{sourdisuniqueness}. For future purposes, we point out that the method of the aforementioned reference yields that
\begin{equation}\label{eqFractal}
  \frac{v_\epsilon(x)}{\eta_\epsilon(x)}<1,\ \ x<\rho_\epsilon.
\end{equation}

Fine estimates for the convergence $\eta_\epsilon\to -\sqrt{\mu^+}$ uniformly as $\epsilon \to 0$ are available from \cite{karali2015ground} (because the positive solution of (\ref{eqPainle}) is nondegenerate, as we have remarked). In particular, $ \eta_\epsilon$ satisfies (\ref{eqSatur}) with $Y=Y_+$.
Essentially we will only need that there exist constants $C,\delta>0$ such that
\begin{equation}\label{eqSimulation}\begin{array}{c}
                                      -C\sqrt{ x+\xi}\leq\eta_\epsilon(x) \leq -\frac{1}{C}\sqrt{x+\xi}, \ \ \epsilon^\frac{2}{3} \leq x+\xi\leq \delta, \\
                                       \\
                                      \eta_\epsilon(x) \sim -\sqrt{\mu'(-\xi)(x+\xi)}, \ \ \epsilon^\frac{2}{3} \ll x+\xi\leq \delta,
                                    \end{array}
\end{equation}
  and that the above relations  can be differentiated with respect to $x$ in the obvious way
(see also Appendix A in \cite{ANS}). For a bit more refined estimates  we refer to Remark \ref{remFormal} below.

Using a well known  trick from \cite{lassoued1999ginzburg}, suppressing the dependence on $\epsilon$, we write
\begin{equation}\label{eqMiro}
  v=\eta w,
\end{equation}
for some $w$ such that
\begin{equation}\label{eqwsign}
  w>0 \ \textrm{for}\ x<\rho;\ \ w<0 \ \textrm{for}\ x>\rho.
\end{equation}
For future purposes, let us note that the upper bound
\[
|v(x)|\leq C\left(\sqrt{x+\xi}+\epsilon^{\frac{1}{3}} \right), \ \ x\in (-\xi, 0),
\]
which follows at once from Lemma 3.1 in \cite{clerc2017theory}, together with the identical lower bound which follows from (\ref{eqSimulation})
imply that
\begin{equation}\label{eqwlower}
|w(x)|\leq C, \ \ x\in (-\xi,0].
\end{equation}

We find that $w$ satisfies
\begin{equation}\label{eqWstr}
 \epsilon^2  w''+2\epsilon^2  \frac{\eta'}{\eta}w'+\eta^2\left(w-w^3 \right)+\epsilon a \frac{f(x)}{\eta}(1-w)=0, \ \ x<0.
\end{equation}
Next, we stretch variables around $x=\rho$, setting
\[
W(y)=w(x)     \ \textrm{where}\ y=\frac{x-\rho}{\tilde{\epsilon}}
\]
with   $\tilde{\epsilon}=\tilde{\epsilon}(\epsilon)>0$ to be determined such that
\begin{equation}\label{eqepsil}
  \tilde{\epsilon}\ll \epsilon^\frac{2}{3}.
\end{equation} We will denote with $\dot{}$ the derivative with respect to $y$.
 We obtain from (\ref{eqWstr}) that
\begin{equation}\label{eqWnew}
\begin{split}
 \ddot{W}+2 \tilde{\epsilon} \frac{\eta'}{\eta}\left(\rho+\tilde{\epsilon}y\right)\dot{W}  +\frac{\tilde{\epsilon}^2}{\epsilon^2}\eta^2\left(\rho+\tilde{\epsilon}y\right)\left(W-W^3 \right)\ \ \ \ \ \ \ \ \ \ \ \ \ \ \ \ \ \ \ \ \ \ &\\
 \ \ \ \ \ \ \ \ \ \ \ \
 +a\frac{\tilde{\epsilon}^2}{\epsilon}  \frac{f}{\eta}\left(\rho+\tilde{\epsilon}y\right)(1-W)=0&
\end{split}
\end{equation}
for $y<-\rho/\tilde{\epsilon}$.
In light of     (\ref{eqcontra}),  (\ref{eqSimulation})  and the working assumption (\ref{eqepsil}), it holds
\begin{equation}\label{eqSim2}\frac{\tilde{\epsilon}^2}{\epsilon^2}\eta^2(\rho+\tilde{\epsilon}y)\sim \mu'(-\xi)\frac{\tilde{\epsilon}^2}{\epsilon^2} (\rho+\xi)\ \ \textrm{as}\ \epsilon\to 0
\end{equation}
(for fixed $y$).
Therefore, we choose
\begin{equation}\label{eqEpsChoice}
  \tilde{\epsilon}=\frac{\epsilon}{\sqrt{\rho+\xi}}.
\end{equation}
Recalling (\ref{eqcontra}), this choice clearly satisfies (\ref{eqepsil}). The point is that we want the third term in (\ref{eqWnew}) to be of the same order as the first one; the other terms will turn out to be of smaller order. Then, equation (\ref{eqWnew}) becomes
\begin{equation}\label{eqWnew2}
\begin{split}
 \ddot{W}+2 \frac{\epsilon}{\sqrt{\rho+\xi}} \frac{\eta'}{\eta}\left(\rho+\tilde{\epsilon}y\right)\dot{W}  +\frac{\eta^2\left(\rho+\tilde{\epsilon}y\right)}{\rho+\xi}\left(W-W^3 \right)\ \ \ \ \ \ \ \ \ \ \ \ \ \ \ \ \ \ \ \ \ \ &\\
 \ \ \ \ \ \ \ \ \ \ \ \
 +a\frac{\epsilon}{\rho+\xi}  \frac{f}{\eta}\left(\rho+\tilde{\epsilon}y\right)(1-W)=0&
\end{split}
\end{equation}

By virtue of (\ref{eqwlower}), we have that $W$ is bounded locally with respect to $\epsilon$.
Hence, using standard elliptic estimates and the usual diagonal argument, keeping in mind (\ref{eqwsign}), passing to a further subsequence if necessary, we find that
\begin{equation}\label{eqCCloca}
  W(y)\to W_0(y)\ \textrm{in}\  C^1_{loc}(\mathbb{R}) \ \textrm{as} \ \epsilon \to 0,
\end{equation}
where $W_0$ satisfies
\[
\frac{1}{\mu'(-\xi)}\ddot{W}_0+W_0-W_0^3=0,\  \ y\in \mathbb{R};
\ \
 yW_0(y)\leq 0,\ \ y\in \mathbb{R}.
\]

We claim that $W_0$ is nontrivial, which would imply that
\begin{equation}\label{eqHyper}
  W_0(y)\equiv -\tanh\left(\sqrt{\frac{\mu'(-\xi)}{2}}y \right).
\end{equation} In passing, we note that the explicit formula for $W_0$ is not important for our purposes, we will essentially use that
\begin{equation}\label{eqEfi}
  W_0(y)\to \mp 1\ \ \textrm{as}\ \ y\to \pm\infty.
\end{equation}
To this end, let $Z_\epsilon$ be the unique positive solution of the following boundary value problem:
\[
 \ddot{Z}+2 \frac{\epsilon}{\sqrt{\rho+\xi}} \frac{\eta'}{\eta}\left(\rho+\tilde{\epsilon}y\right)\dot{Z}  +\frac{\eta^2\left(\rho+\tilde{\epsilon}y\right)}{\rho+\xi}\left(Z-Z^3 \right)=0,\ \ y\in \left(\frac{-\xi+\epsilon^\frac{2}{3}-\rho}{\tilde{\epsilon}} ,0\right);
\]
\[
Z\left(\frac{-\xi+\epsilon^\frac{2}{3}-\rho}{\tilde{\epsilon}} \right)=\frac{1}{4},\ \ \ Z(0)=0.
\]As will become apparent  shortly, the specific value $1/4$ is not of importance.
The existence of such a $Z_\epsilon$ follows by directly minimizing the associated energy functional
\begin{equation}\label{eqEepsil}
  E_\epsilon(Z)=\int_{(-\xi+\epsilon^\frac{2}{3}-\rho)\tilde{\epsilon}^{-1}}^{0}\left\{\frac{\eta^2\left(\rho+\tilde{\epsilon}y\right)}{2}(\dot{Z})^2+
\frac{\eta^4\left(\rho+\tilde{\epsilon}y\right)}{\rho+\xi}\frac{\left(1-Z^2\right)^2}{4} \right\}dy
\end{equation}
subject to the above boundary conditions. Moreover, simple energy considerations give that $0<Z_\epsilon<1$.
The uniqueness comes from the observation that   $\{tZ_\epsilon \ :\  t\in (0,1)\}$  is a family of lower solutions to the above boundary value problem and Serrin's sweeping principle (see \cite{sourdisuniqueness} and the references therein). In fact, by virtue of (\ref{eqAssump}), we note that these are also lower solutions to the equation (\ref{eqWnew2}) in the same interval. As a consequence of the assumption (\ref{eqcontra}), the blow-up analysis in \cite{clerc2017theory} yields that $ \eta_\epsilon$ and $v_\epsilon$ share the same first order term in their corresponding inner expansions around $-\xi$. Hence, both $\eta_\epsilon$ and $v_\epsilon$ satisfy (\ref{eqSatur}) with $Y=-Y_+$.
In particular, given any $L>0$,  it holds
\begin{equation}\label{eqInfinity}
  W_\epsilon\left(\frac{-\xi-\rho+L\epsilon^\frac{2}{3}}{\tilde{\epsilon}}\right)\to 1 \ \ \textrm{as}\ \ \epsilon\to 0.
\end{equation}
Thus, since  $W_\epsilon\left((-\xi+\epsilon^\frac{2}{3}-\rho)\tilde{\epsilon}^{-1} \right)>1/2$, $W_\epsilon(0)=0$ and  $W_\epsilon>0$ in between, we deduce by Serrin's sweeping principle that
\[
W_\epsilon(y)>Z_\epsilon(y),\ \ y\in \left((-\xi+\epsilon^\frac{2}{3}-\rho)\tilde{\epsilon}^{-1},0 \right).
\]
On the other hand, as in Proposition 2.3 in \cite{nakashima2003multi}, we have
\[
Z_\epsilon\to -\tanh\left(\sqrt{\frac{\mu'(-\xi)}{2}}y \right)\ \ \textrm{in}\ \ C^1_{loc}\left((-\infty,0]\right)\ \ \textrm{as}\ \ \epsilon\to 0.
\]
The above two relations imply that $W_0$ is nontrivial, and thus is given by (\ref{eqHyper})  as claimed.


Let us consider now (\ref{eqWnew2}) without derivatives, which recalling (\ref{eqEpsChoice}) and after dividing by $1-W$ becomes
\[
\Sigma^2+\Sigma+a\epsilon \frac{f}{\eta^3}(\rho+\tilde{\epsilon}y)=0.
\]
The above algebraic equation can be solved explicitly and has the following two  solutions:
\[
\Sigma_\pm(y)=\frac{-1\pm \sqrt{1-4a\epsilon \frac{f}{\eta^3}(\rho+\tilde{\epsilon}y)}}{2}.
\]
It follows from (\ref{eqcontra}) and (\ref{eqSimulation}) that there is some fixed large $D>0$ such that these are real valued  for  
\[-\xi+ D\epsilon^\frac{2}{3}\leq \rho+\tilde{\epsilon}y\leq \delta,\] provided
that $\epsilon>0$ is sufficiently small.
From now on let us fix a $D$ such that
\begin{equation}\label{eqSigmata}
  0<1+\Sigma_-(y) <\frac{1}{100}\ \ \textrm{and}\ \  -\frac{1}{100}<\Sigma_+(y) <0
\end{equation}
in the aforementioned interval, for sufficiently small $\varepsilon>0$ (the number 100 has no significance here).

Let
\begin{equation}\label{eqPhidefi}
  \Phi(y)=1-W(y),\ \ y\in  \left[D\sqrt{\rho+\xi}\epsilon^{-\frac{1}{3}}-\frac{(\rho+\xi)^\frac{3}{2}}{\epsilon},0\right].
\end{equation}
We know that \begin{equation}\label{eqphiPosu}
               0<\Phi<1
             \end{equation}
(recall (\ref{eqFractal})).
We can write (\ref{eqWnew2}) as
\begin{equation}\label{eq1001}
  -\ddot{\Phi}-2\frac{\epsilon}{\sqrt{\rho+\xi}}\frac{\eta'}{\eta}\left(\rho+\tilde{\epsilon}y\right)\dot{\Phi}+\frac{\eta^2(\rho+\tilde{\epsilon}y)}{\rho+\xi}
  \left(W-\Sigma_-(y) \right)\left(W-\Sigma_+(y) \right)\Phi=0.
\end{equation}
By virtue of (\ref{eqCCloca}), (\ref{eqEfi}), (\ref{eqInfinity}) with $L=D$,  and (\ref{eqSigmata}), we deduce from (\ref{eqWnew2}) via the maximum principle that there exists an $M>0$ such that\begin{equation}\label{eqGeq12}
  0<1-W(y)<\frac{1}{100},\ \ y\in  \left[D\sqrt{\rho+\xi}\epsilon^{-\frac{1}{3}}-\frac{(\rho+\xi)^\frac{3}{2}}{\epsilon},-M\right],
\end{equation}
provided that $\epsilon>0$ is sufficiently small. More precisely, we observe that $\underline{W}_t=1-t$, $\frac{1}{100}<t\leq1$, is a family of lower solutions to (\ref{eqWnew2}) which allows  to use Serrin's sweeping principle since $W>0$, see \cite{sourdisuniqueness} and the references therein. Concerning the second term in (\ref{eq1001}), it follows from (\ref{eqSimulation}) that
\begin{equation}\label{eqMiddle}
  0<\frac{\epsilon}{\sqrt{\rho+\xi}}\frac{\eta'}{\eta}\left(\rho+\tilde{\epsilon}y\right)\leq C \frac{\epsilon^\frac{1}{3}}{\sqrt{\rho+\xi}},\ \ y\in\left[D\sqrt{\rho+\xi}\epsilon^{-\frac{1}{3}}-\frac{(\rho+\xi)^\frac{3}{2}}{\epsilon},0\right];
\end{equation}
  concerning the last term, we obtain from (\ref{eqSimulation}) and (\ref{eqGeq12}) that
\begin{equation}\label{eqStrong}
C\geq \frac{\eta^2(\rho+\tilde{\epsilon}y)}{\rho+\xi}
  \left(W-\Sigma_-  \right)\left(W-\Sigma_+  \right)\geq \frac{1}{C},\ \ y\in\left[-\frac{2(\rho+\xi)^\frac{3}{2}}{3\epsilon},-M\right].
\end{equation}
  So, by a standard barrier argument, we deduce from   (\ref{eqphiPosu}), (\ref{eq1001}), (\ref{eqMiddle}) and the above relation that
\[
0<  \Phi(y)\leq  e^{cy}+e^{-c\left(y+\frac{2(\rho+\xi)^\frac{3}{2}}{3\epsilon} \right)},\ \  y\in\left[-\frac{2(\rho+\xi)^\frac{3}{2}}{3\epsilon},-M\right].
\]
Consequently, by recalling the definition of $\Phi$ from (\ref{eqPhidefi}),   we infer that
\begin{equation}\label{eqPhiOuttyStrongest}
 0<1-W(y)\leq   e^{cy}+ e^{-c\left\{\frac{\rho+\xi}{\epsilon^\frac{2}{3}}\right\}}   ,\ \  y\in\left[-\frac{(\rho+\xi)^\frac{3}{2}}{2\epsilon}-1,0\right].
\end{equation}
In turn, by standard elliptic estimates, via (\ref{eq1001}),   (\ref{eqMiddle}) and the upper bound in (\ref{eqStrong}),  we get that
\begin{equation}\label{eqPhiOuttyStrongestDeriva}
 \left|\dot{W}(y)\right|\leq C e^{cy}+Ce^{-c\left\{\frac{\rho+\xi}{\epsilon^\frac{2}{3}}\right\}} ,\ \  y\in\left[-\frac{(\rho+\xi)^\frac{3}{2}}{2\epsilon},0\right].
\end{equation}

Let us    write (\ref{eqWnew2}) as
\begin{equation}\label{eqWnew2est}
 \ddot{W}+2\frac{\epsilon}{\sqrt{\rho+\xi}}\frac{\eta'}{\eta}\left(\rho+\tilde{\epsilon}y\right)\dot{W}+\frac{\partial}{\partial W}G(W,y)=0,  \end{equation}
with
\begin{equation}\label{eqG}
  G(W,y)=-\frac{\eta^2(\rho+\tilde{\epsilon}y)}{\rho+\xi}\frac{\left(1-W^2 \right)^2}{4} -a\frac{\epsilon}{\rho+\xi}  \frac{f}{\eta}\left(\rho+\tilde{\epsilon}y\right)\frac{(1-W)^2}{2}.
\end{equation}
When multiplied by  $\dot{W}$,  equation (\ref{eqWnew2est}) reads as
\begin{equation}\label{eqBasic}
\begin{split}
\frac{d}{dy}\left(\frac{(\dot{W})^2}{2} \right)+2\frac{\epsilon}{\sqrt{\rho+\xi}}\frac{\eta'}{\eta}\left(\rho+\tilde{\epsilon}y\right)(\dot{W})^2+\frac{d}{dy}\left(G(W,y)\right)  \ \ \ \ \ \ \ \ \ \ \ \ \ \ \ \ \ \ \ \ \ \ &\\
 \ \ \  + \frac{\epsilon}{(\rho+\xi)^\frac{3}{2}}\eta\eta'\left(\rho+\tilde{\epsilon}y\right)\frac{\left(1-W^2 \right)^2}{2}\ \ \ \ \ \ \ \ \ \ \  \ \ \ \ \ \ \ \ \ \ \ \ \ \ \ \ \ \ \ \ \ \   \  & \\
 -a\frac{\epsilon^2}{(\rho+\xi)^\frac{3}{2}}  \frac{f\eta'}{\eta^2}\left(\rho+\tilde{\epsilon}y\right)\frac{(1-W)^2}{2}+a\frac{\epsilon^2}{(\rho+\xi)^\frac{3}{2}}  \frac{f'}{\eta}\left(\rho+\tilde{\epsilon}y\right)\frac{(1-W)^2}{2}=0. \ \ \ \ \  \ \  \ \ \ \ \ \ \ \ &
  \end{split}
\end{equation}

We shall next integrate the above relation over \[y\in I=\left(-\frac{(\rho+\xi)^\frac{3}{2}}{2\epsilon},(\delta-\rho)\frac{\sqrt{\rho+\xi}}{\epsilon}\right)=(\alpha,\beta).\]
The following fact will be useful in the sequel. By the definition  (\ref{eqMiro}) of $w$, and standard estimates for $v$ and $\eta$ that hold uniformly away from their corner layer at $-\xi$ (see Proposition 2.1 in \cite{ignat2006critical}), we infer that
  \begin{equation}\label{eqlokatzis}
    \|w(x)-\Sigma_-(\frac{x-\rho}{\tilde{\epsilon}})\|_{C^1\left[-\xi+\frac{\delta}{2},-\xi+2\delta\right]}\leq C\epsilon^2.
  \end{equation}
The integral of the first term in (\ref{eqBasic}) is plainly the half of
\[\begin{array}{rcl}
                                    (\dot{W})^2\left(\beta \right) -(\dot{W})^2\left(\alpha \right) & = & \frac{\epsilon^2}{\rho+\xi}  (w')^2\left(-\xi+\delta  \right)-(\dot{W})^2\left(-\frac{(\rho+\xi)^\frac{3}{2}}{2\epsilon} \right) \\
                                      &   &   \\
                                      & = & O\left(
\frac{\epsilon^4}{\rho+\xi}\right)+O\left(\left(
\frac{\epsilon^\frac{2}{3}}{\rho+\xi}\right)^{10}\right) \\
                                     &  &  \\
                                     & = &O\left(\left(
\frac{\epsilon^\frac{2}{3}}{\rho+\xi}\right)^{6}\right),
                                  \end{array}
\]
where we used (\ref{eqPhiOuttyStrongestDeriva}),  and (\ref{eqlokatzis}) which implies that  $w'\left(-\xi+\delta\right)=O(\epsilon)$ (recall also the formula for $\Sigma_-$).
We point out in passing that the rough estimate $w'(-\xi+\delta)=O(1)$ would have actually sufficed for our purposes. Regarding the second term in (\ref{eqBasic}), the fact that it is nonnegative will suffice for the time being.
Using (\ref{eqPhiOuttyStrongest}) and (\ref{eqlokatzis}), we find from the definition of $G$ in (\ref{eqG}) that  the integral of the third term in (\ref{eqBasic}) can be estimated as follows:
\[
\int_{I}^{}\frac{d}{dy}\left(G(W,y)\right)dy=G\left(W(\beta),\beta \right)-G\left(W(\alpha),\alpha \right)  =O\left(
\frac{\epsilon}{\rho+\xi}\right).
\]
Concerning the     fourth term in (\ref{eqBasic}),
thanks to (\ref{eqSimulation}), we find that
\[
\int_{I}^{}\eta\eta'\left(\rho+\tilde{\epsilon}y\right)\frac{\left(1-W^2 \right)^2}{2}dy\geq c\int_{I}^{} \left(1-W^2 \right)^2 dy.
\]
Concerning the fifth term in (\ref{eqBasic}), recalling (\ref{eqwlower})  and using once more (\ref{eqSimulation}),   we have
\begin{equation}\label{eqBigThree}\begin{array}{rcl}
    \left|\int_{I}^{}  \frac{f\eta'}{\eta^2}\left(\rho+\tilde{\epsilon}y\right)\frac{(1-W)^2}{2}dy\right| & \leq &  C \int_{I}^{ }  \left(\rho+\xi+\tilde{\epsilon}y\right)^{-\frac{3}{2}}dy \\
      &   &   \\
      & \leq & C\tilde{\epsilon}^{-1} (\rho+\xi)^{-\frac{1}{2}}   \\
      &   &   \\
      & \leq  & \frac{C}{\epsilon}.
  \end{array}
 \end{equation}
Finally, concerning the last term in (\ref{eqBasic}), we obtain similarly that
\[
\left|\int_{I}^{}\frac{f'}{\eta}\left(\rho+\tilde{\epsilon}y\right)\frac{(1-W)^2}{2}dy\right|\leq \frac{C}{\epsilon}.
\]
Collecting all of the above, we infer from the integration of (\ref{eqBasic}) over $I$ that
\begin{equation}\label{eqL222}
  \int_{I}^{} \left(1-W^2 \right)^2 dy\leq C.
\end{equation}

We claim that (\ref{eqL222}) implies that there exist $c,N>0$ such that
\begin{equation}\label{eqClaimaraTelos}
  W(y)\leq-c,\ \ \ y\in \left[N,5\frac{(\rho+\xi)^\frac{3}{2}}{\epsilon} \right],
\end{equation}if $\epsilon>0$ is sufficiently small.
Indeed, if not, then there would exist points \begin{equation}\label{eqCrucial}p_\epsilon\in\left(0, 5\frac{(\rho+\xi)^\frac{3}{2}}{\epsilon}\right]\ \
\textrm{with}\ \ p_\epsilon\to +\infty
\end{equation} such that
\[
W(p_\epsilon)\to 0.
\]
Let
\[
\tilde{W}(y)=W(y+p).
\]
Then, we have
\begin{equation}\label{eqWnew2!!!}
\begin{split}
 \ddot{\tilde{W}}+2 \frac{\epsilon}{\sqrt{\rho+\xi}} \frac{\eta'}{\eta}\left(\rho+\tilde{\epsilon}p+\tilde{\epsilon}y\right)\dot{\tilde{W}}  +\frac{\eta^2\left(\rho+\tilde{\epsilon}p+\tilde{\epsilon}y\right)}{\rho+\xi}\left(\tilde{W}-\tilde{W}^3 \right)\ \ \ \   \ \ \ \ \ \ &\\
 \ \ \ \ \ \
 +a\frac{\epsilon}{\rho+\xi}  \frac{f}{\eta}\left(\rho+\tilde{\epsilon}p+\tilde{\epsilon}y\right)(1-\tilde{W})=0&
\end{split}
\end{equation}
(at least) for $y\in\left[-1,100\frac{(\rho+\xi)^\frac{3}{2}}{\epsilon} \right]$, $\tilde{W}< 0$ therein, and
$
\tilde{W}(0)\to 0.
$
Then, since $\tilde{\epsilon}p\to 0$, we can argue as we did for   (\ref{eqCCloca}) to find that
$
\tilde{W} \to \tilde{W}_0$  in  $C^1_{loc}\left([-1,\infty)\right),
$
where
\[
\frac{1}{\mu'(-\xi)}\ddot{\tilde{W}}_0+\tilde{W}_0-\tilde{W}_0^3=0,\  \ y>-1;
\ \
 \tilde{W}_0(y)\leq 0,\ \ y>-1,\ \ \tilde{W}_0(0)=0.
\]
It is clear that $\tilde{W}_0\equiv 0$. So, given any $K>1$, it holds
\[
\int_{I}^{}\left(1-W^2 \right)^2 dy\geq\int_{p}^{p+K}\left(1-W^2 \right)^2 dy=\int_{0}^{K}\left(1-\tilde{W}^2 \right)^2 dy\geq \frac{K}{2},
\]
provided that $\epsilon>0$ is sufficiently small, which contradicts (\ref{eqL222}) and establishes the validity of (\ref{eqClaimaraTelos}).

The estimate (\ref{eqL222}) yields that
\[
\int_{5\frac{(\rho+\xi)^\frac{3}{2}}{\epsilon}}^{10\frac{(\rho+\xi)^\frac{3}{2}}{\epsilon}}\left(1-W^2 \right)^2 dy\leq C,
\]
which implies that
\begin{equation}\label{eqclaimara3333}
\left|-1-W(q)\right|<\frac{1}{100}\ \ \textrm{for\ some}\ \ q\in \left(5\frac{(\rho+\xi)^\frac{3}{2}}{\epsilon},10\frac{(\rho+\xi)^\frac{3}{2}}{\epsilon}\right).
\end{equation}

We may further assume that the  $N$ in (\ref{eqClaimaraTelos}) is   such that \[-1<W_0(N)<-1+\frac{1}{1000},\] (recall (\ref{eqEfi})).
Then, thanks to (\ref{eqCCloca}), (\ref{eqSigmata}) and (\ref{eqclaimara3333}), as we did for (\ref{eqGeq12}) we have that
\[
W(y)>-1- \frac{1}{100},\ \ y\in\left[N, 5\frac{(\rho+\xi)^\frac{3}{2}}{\epsilon}\right].
\]
On the other side, the relations (\ref{eqClaimaraTelos}) and (\ref{eqclaimara3333}) allow us to apply the same argument in the opposite direction, and thus arrive at
\begin{equation}\label{eqFinaleee}
\left|W(y)+1\right|<\frac{1}{100},\ \ y\in\left[N, 5\frac{(\rho+\xi)^\frac{3}{2}}{\epsilon}\right].
\end{equation}

Let now
\[
  \Psi(y)=W(y)-\Sigma_-(y),\ \ y\in  \left[0, 5\frac{(\rho+\xi)^\frac{3}{2}}{\epsilon}\right].
\]
 We can write (\ref{eqWnew2}) as
\begin{equation}\label{eq1002}
\begin{split}
  \ddot{\Psi}+2\frac{\epsilon}{\sqrt{\rho+\xi}}\frac{\eta'}{\eta}\left(\rho+\tilde{\epsilon}y\right)\dot{\Psi}+\frac{\eta^2(\rho+\tilde{\epsilon}y)}{\rho+\xi}
  \left(1-W  \right)\left(W-\Sigma_+(y) \right)\Psi=\ \ \ \ \ \ \ \ \ \ \ \ \ \ \ \ & \ \ \ \ \ \ \ \ \ \ \ \ \ \ \ \  \\ \ \ \ \ \ \ \ \ \ \ \ \ \ \ \ \ \ \ \
 \ \ \ \ \ \ \ \ \ \       -\ddot{\Sigma}_--2\frac{\epsilon}{\sqrt{\rho+\xi}}\frac{\eta'}{\eta}\left(\rho+\tilde{\epsilon}y\right)\dot{\Sigma}_-.&
\end{split}
 \end{equation}
 We note that
\begin{equation}\label{eqs1}
\left|-1-\Sigma_-(y)\right|\leq C\epsilon (\rho+\xi+\tilde{\epsilon}y)^{-\frac{3}{2}},
\end{equation}
\begin{equation}\label{eqs2}
  \left|\dot{\Sigma}_-(y)\right|\leq C\frac{\epsilon^2}{\sqrt{\rho+\xi}} (\rho+\xi+\tilde{\epsilon}y)^{-\frac{5}{2}},\ \
\left|\ddot{\Sigma}_-(y)\right|\leq C\frac{\epsilon^3}{\rho+\xi} (\rho+\xi+\tilde{\epsilon}y)^{-\frac{7}{2}},
\end{equation}
for $y\in \left[0,5\frac{(\rho+\xi)^\frac{3}{2}}{\epsilon}\right]$.
Thus,  since
\begin{equation}\label{eqMiddleNew}
  0<\frac{\eta'}{\eta}\left(\rho+\tilde{\epsilon}y\right)\leq \frac{C}{\rho+\xi},\ \ y\in \left[-\frac{(\rho+\xi)^\frac{3}{2}}{2\epsilon},5\frac{(\rho+\xi)^\frac{3}{2}}{\epsilon}\right],
\end{equation}
  equation (\ref{eq1002}) becomes
\begin{equation}\label{eq1003}
\begin{split}
  \ddot{\Psi}+2\frac{\epsilon}{\sqrt{\rho+\xi}}\frac{\eta'}{\eta}\left(\rho+\tilde{\epsilon}y\right)\dot{\Psi}+\frac{\eta^2(\rho+\tilde{\epsilon}y)}{\rho+\xi}
  \left(1-W  \right)\left(W-\Sigma_+(y) \right)\Psi=\ \ \ \ \ \ \ \ \ \ \ \ \ &\\ \ \ \ \ \ \ \  \ O\left(\frac{\epsilon^3}{\rho+\xi} (\rho+\xi+\tilde{\epsilon}y)^{-\frac{7}{2}}\right)&,\end{split}
 \end{equation}
for $y\in  \left[0, 5\frac{(\rho+\xi)^\frac{3}{2}}{\epsilon}\right]$.
By virtue of (\ref{eqFinaleee}), as we did previously for $\Phi$ (also keep in mind (\ref{eqwlower})), we get that
\[
\left|W-\Sigma_-\right|\leq C\frac{\epsilon^3}{(\rho+\xi)^{\frac{9}{2}}} +Ce^{-cy}+Ce^{c\left(y-5\frac{(\rho+\xi)^\frac{3}{2}}{\epsilon} \right)},\ \ y\in \left[0, 5\frac{(\rho+\xi)^\frac{3}{2}}{\epsilon}\right].
\]
Thus, it holds
\begin{equation}\label{eq11}
\left|W-\Sigma_-\right|\leq C\frac{\epsilon^3}{(\rho+\xi)^{\frac{9}{2}}} +Ce^{-cy},\ \ y\in \left[0, 4\frac{(\rho+\xi)^\frac{3}{2}}{\epsilon}\right].
\end{equation}
In turn, by using the upper bound\[
\frac{\eta^2(\rho+\tilde{\epsilon}y)}{\rho+\xi}\leq C,\ \ y\leq 5\frac{(\rho+\xi)^\frac{3}{2}}{\epsilon},
\]
we deduce from  (\ref{eqMiddleNew}), (\ref{eq1003}), (\ref{eq11}), and standard elliptic estimates that
\begin{equation}\label{eqAboveG}
\left| \dot{W}-\dot{\Sigma}_-\right|\leq C\frac{\epsilon^3}{(\rho+\xi)^\frac{9}{2}} +Ce^{-cy}, \ \ y\in \left[0,\frac{(\rho+\xi)^{\frac{3}{2}}}{\epsilon } \right].\end{equation}
By combining (\ref{eqs1}) and (\ref{eq11}), we find that
\begin{equation}\label{eqW+1}
\left|W+1\right|\leq C\frac{\epsilon }{(\rho+\xi)^{\frac{3}{2}}} +Ce^{-cy} ,\ \ y\in \left[0, \frac{(\rho+\xi)^\frac{3}{2}}{\epsilon}\right].
\end{equation}
By combining (\ref{eqs2}) and (\ref{eqAboveG}), we get that
\begin{equation}\label{eqW+1Deriva}
\left| \dot{W}\right|\leq C\frac{\epsilon^2}{(\rho+\xi)^3} +Ce^{-cy}, \ \ y\in \left[0, \frac{(\rho+\xi)^\frac{3}{2}}{\epsilon}\right].\end{equation}

We shall integrate the above relation over \[y\in J=\left(-\frac{(\rho+\xi)^\frac{3}{2}}{N\epsilon},\frac{(\rho+\xi)^\frac{3}{2}}{N\epsilon}\right)=(\tilde{\alpha},\tilde{\beta}),\]with $N$ to be chosen sufficiently large but independent of $\epsilon$. Below we will carefully estimate   each term in the resulting relation.
Thanks to (\ref{eqPhiOuttyStrongestDeriva}) and (\ref{eqW+1Deriva}), in regards to the first term in (\ref{eqBasic}), we find that
\[
     (\dot{W})^2\left(\frac{(\rho+\xi)^\frac{3}{2}}{N\epsilon} \right) -(\dot{W})^2\left(-\frac{(\rho+\xi)^\frac{3}{2}}{N\epsilon} \right)  =
    O\left(\frac{\epsilon^4}{(\rho+\xi)^6}\right) ,
\] as $\epsilon\to 0$.
Regarding the second term in (\ref{eqBasic}), thanks to (\ref{eqSimulation}) and (\ref{eqCCloca}), we have
\[
0<\frac{(\rho+\xi)^\frac{3}{2}}{\epsilon}\frac{\epsilon}{\sqrt{\rho+\xi}}\frac{\eta'}{\eta}\left(\rho+\tilde{\epsilon}y\right)(\dot{W})^2\to \frac{1}{2}(\dot{W}_0)^2\  \textrm{in} \ C_{loc}(\mathbb{R})\ \textrm{as}\ \epsilon\to 0.
\]
Therefore, we obtain from (\ref{eqPhiOuttyStrongestDeriva}),  (\ref{eqMiddleNew}),   (\ref{eqW+1Deriva}), and Lebesgue's dominated convergence theorem that
\begin{equation}\label{eqend1}
\frac{(\rho+\xi)^\frac{3}{2}}{\epsilon}\frac{\epsilon}{\sqrt{\rho+\xi}}\int_{J}^{}\frac{\eta'}{\eta}\left(\rho+\tilde{\epsilon}y\right)(\dot{W})^2dy\to \frac{1}{2}\int_{-\infty}^{\infty}(\dot{W}_0)^2dy.
\end{equation}
Recalling (\ref{eqPhiOuttyStrongest}), (\ref{eqG}) and (\ref{eqW+1}),  the integral of the third term in (\ref{eqBasic}) can be estimated   as follows:
\[\begin{array}{rcl}
    \int_{J}^{}\frac{d}{dy}\left(G(W,y)\right)dy & = & G\left(W(\tilde{\beta}),\tilde{\beta} \right)-G\left(W(\tilde{\alpha}),\tilde{\alpha} \right) \\
      &   &  \\
      & = & \frac{2a}{\sqrt{1+1/N}}\frac{f(-\xi)}{\sqrt{\mu'(-\xi)}}\frac{\epsilon}{(\rho+\xi)^\frac{3}{2}}  \left(1+o(1)\right),
  \end{array}
\]as $\epsilon\to 0$.
Concerning the     fourth term in (\ref{eqBasic}), by working   as we did for (\ref{eqend1}), using (\ref{eqPhiOuttyStrongest}) and (\ref{eqW+1}), we find that
\[
\int_{J}^{}\eta\eta'\left(\rho+\tilde{\epsilon}y\right)\frac{\left(1-W^2 \right)^2}{2}dy\to \frac{\mu'(-\xi)}{4}\int_{-\infty}^{\infty} \left(1-W_0^2 \right)^2 dy=\frac{1}{2}\int_{-\infty}^{\infty}(\dot{W}_0)^2dy.
\]
Concerning the fifth term in (\ref{eqBasic}),  keeping in mind    (\ref{eqSimulation}) and (\ref{eqwlower}),   we have
\begin{equation}\label{eqformalok}
  \begin{array}{rcl}
    \int_{J}^{}  \frac{f\eta'}{\eta^2}\left(\rho+\tilde{\epsilon}y\right)\frac{(1-W)^2}{2}dy & = &   \int_{-\frac{(\rho+\xi)^\frac{3}{2}}{N\epsilon}}^{ \frac{(\rho+\xi)^\frac{3}{2}}{N\epsilon}}  \frac{f\eta'}{\eta^2}\left(\rho+\tilde{\epsilon}y\right)\frac{(1-W)^2}{2}dy \\
      &   &   \\
      & \leq & \frac{C}{(\rho+\xi)^\frac{3}{2}}  \int_{-\frac{(\rho+\xi)^\frac{3}{2}}{N\epsilon}}^{ \frac{(\rho+\xi)^\frac{3}{2}}{N\epsilon}} 1 dy \\
      &   &   \\
      & \leq  & \frac{C}{N\epsilon},
  \end{array}
\end{equation}where the above constant $C$ is independent of both large $N$ and small $\epsilon$.
Concerning the last term in (\ref{eqBasic}), by using (\ref{eqwlower})  and that $-\eta\geq c\sqrt{\rho+\xi}$ in $J$, we get that
\[
\left|\int_{J}^{}\frac{f'}{\eta}\left(\rho+\tilde{\epsilon}y\right)\frac{(1-W)^2}{2}dy\right|\leq C\frac{\rho+\xi}{\epsilon}.
\]
Putting all the above in the integral of (\ref{eqBasic}) over $J$ gives us that
\begin{equation}\label{eqFinaleGGGG}
  \frac{\epsilon}{(\rho+\xi)^\frac{3}{2}}\left|\frac{3}{2}\int_{-\infty}^{\infty}(\dot{W}_0)^2dy+\frac{2a}{\sqrt{1+1/N}}\frac{f(-\xi)}{\sqrt{\mu'(-\xi)}}+o(1) \right|\leq C\frac{\epsilon}{\sqrt{\rho+\xi}}+\frac{C}{N}\frac{\epsilon}{(\rho+\xi)^\frac{3}{2}},
\end{equation}
as $\epsilon\to 0$, where the above constant $C$ is independent of both large $N$ and small $\epsilon$.
By letting first $\epsilon\to 0$ and then $N\to \infty$ in the above relation, we infer that
\[
\frac{3}{2}\int_{-\infty}^{\infty}(\dot{W}_0)^2dy+2a\frac{f(-\xi)}{\sqrt{\mu'(-\xi)}}=0.
\]
We note that
\[
\int_{-\infty}^{\infty}(\dot{W}_0)^2dy=-\frac{\sqrt{\mu'(-\xi)}}{\sqrt{2}}\int_{-\infty}^{\infty}(1-W_0^2)\dot{W}_0dy=\sqrt{2}\sqrt{\mu'(-\xi)}.
\]
So, we have been led to a contradiction unless
\begin{equation}\label{eqaa}a= -\frac{3}{2\sqrt{2}}\frac{\mu'(-\xi)}{f(-\xi)}.\end{equation}
Therefore, to complete the proof, it just remains to exclude the above possibility.
To this end, we will estimate the integral in (\ref{eqBigThree}) over $J\subset I$ from below with, say,  $N=2$ in the definition of the interval
$J$. We find as above that
\[
  \begin{array}{rcl}
    \int_{J}^{}  \frac{f\eta'}{\eta^2}\left(\rho+\tilde{\epsilon}y\right)\frac{(1-W)^2}{2}dy & \geq &   \int_{0}^{ \frac{(\rho+\xi)^\frac{3}{2}}{2\epsilon}}  \frac{f\eta'}{\eta^2}\left(\rho+\tilde{\epsilon}y\right)\frac{(1-W)^2}{2}dy \\
      &   &   \\
      & \geq & \frac{c}{(\rho+\xi)^\frac{3}{2}}  \int_{0}^{\frac{(\rho+\xi)^\frac{3}{2}}{2\epsilon}} 1 dy \\
      &   &   \\
      & \geq  & \frac{c}{\epsilon}.
  \end{array}
\]
Hence, keeping in mind the upper bound in (\ref{eqBigThree}) and that $J\subset  I$, we get that
\[
\epsilon\int_{J}^{}  \frac{f\eta'}{\eta^2}\left(\rho+\tilde{\epsilon}y\right)\frac{(1-W)^2}{2}dy\to c_*
\]
for some $c_*\in (0,\infty)$.
Now, assuming to the contrary that $a$ was as in (\ref{eqaa}), instead of (\ref{eqFinaleGGGG}) we end up with
\[
\frac{\epsilon}{(\rho+\xi)^\frac{3}{2}}c_*= O\left(\frac{\epsilon}{\sqrt{\rho+\xi}}\right)+\frac{o(\epsilon)}{(\rho+\xi)^\frac{3}{2}},
\]
as $\epsilon\to 0$. Clearly we have been led again to a contradiction, which completes the proof of the theorem.

\end{proof}
\begin{rem}\label{remFormal}Let us give the formal argument that led us to the  rigorous  analysis
    following (\ref{eqCCloca}). According to folklore, we write
\[
W(y)=W_0(y)+\hat{\epsilon}W_1(y)+\textrm{higher\ order\ terms},
\]
and try to find $\hat{\epsilon}$ and $W_1$ by plugging this into (\ref{eqWnew2}).

Firstly, we will need the following estimates for $-\eta$ near $-\xi$ that follow in analogy to \cite[Thm. 1.1]{karali2015ground}. It holds
\[
  \eta_\epsilon(x)=\nu_\epsilon(x)+ O\left(\epsilon^2\right)x^{-\frac{5}{2}}, \   x\in (-\xi+\epsilon^\frac{2}{3},-\xi+\delta),
\]
uniformly as $\epsilon\to 0$, for some  fixed small $\delta>0$, where $\nu_\epsilon<0$ solves
the algebraic equation (\ref{eqnueps})
near $-\xi$.
It follows readily that
\[
\nu_\epsilon(x)= -\sqrt{\mu'(-\xi)(x+\xi)}+\frac{af(-\xi)}{2\mu'(-\xi)}\frac{\epsilon}{x+\xi}+O\left(\epsilon\right).
\]

The above together with the relations    (\ref{eqcontra}) and (\ref{eqSimulation})  yield the following:
\[
\begin{split}
   \frac{1}{\rho+\xi}\eta^2\left( \rho+\tilde{\epsilon}y\right)  =  \mu'(-\xi)+\mu'(-\xi)\frac{\epsilon}{(\rho+\xi)^\frac{3}{2}}y-\frac{af(-\xi)}{\sqrt{\mu'(-\xi)}}\frac{\epsilon}{(\rho+\xi)^\frac{3}{2}} &   \\
    +o\left(\frac{\epsilon}{(\rho+\xi)^\frac{3}{2}}\right)(|y|+1),  &
\end{split}
    \]
\[
  \frac{\epsilon }{\sqrt{\rho+\xi}} \frac{\eta'}{\eta}\left(\rho+\tilde{\epsilon}y\right)  =  \frac{\epsilon}{2(\rho+\xi)^\frac{3}{2}} +o\left(\frac{\epsilon}{(\rho+\xi)^\frac{3}{2}}\right)(|y|+1),
    \]
\[
  \frac{\epsilon}{\rho+\xi}  \frac{f}{\eta}\left(\rho+\tilde{\epsilon}y\right)  =  -\frac{f(-\xi)}{\sqrt{\mu'(-\xi)}}\frac{\epsilon}{(\rho+\xi)^\frac{3}{2}} +o\left(\frac{\epsilon}{(\rho+\xi)^\frac{3}{2}}\right)(|y|+1),
    \]
uniformly for $|\tilde{\epsilon}y|\leq \delta$, as $\epsilon\to 0$.



Taking the above into account, we find that
\[
\hat{\epsilon}=\frac{\epsilon}{\rho+\xi}
\]
and
\[\begin{split}
     \ddot{W}_1+(1-3W_0^2)W_1= & -\frac{\mu'(-\xi)}{2}y(W_0-W_0^3)-(\dot{W}_0)^2 \\
       &+\frac{af(-\xi)}{\sqrt{\mu'(-\xi)}}(W_0-W_0^3) +\frac{af(-\xi)}{\sqrt{\mu'(-\xi)}}(1-W_0)
  \end{split}
\]
for $y\in \mathbb{R}$.
However, the last term does not decay to zero as $y\to +\infty$, which is necessary in order to get a solution $W_1$ that decays to zero as $|y|\to \infty$.
\end{rem}
\section{Proof of Corollary \ref{cor3}}
\begin{proof}
  In view of (\ref{eqastar}) and (\ref{eqAlpha}), for $\alpha \in (\alpha_*,0)$ with
  \[
  \alpha_*=\alpha_*(\xi,\mu,f)=\frac{a_*f(-\xi)}{\sqrt{2}\mu'(-\xi)}=\frac{f(-\xi)}{3\mu'(-\xi)}\inf_{x\in (-\xi,0]}\frac{\left(\mu(x)\right)^\frac{3}{2}}{\int_{-\xi}^{x}|f|\sqrt{\mu}},\]our Theorem \ref{thm1} implies the existence of a solution to (\ref{eqPainle}), given by appropriate reflections and modulations of $Y$ from (\ref{eqSatur}), that changes sign exactly once and is energy
  minimizing with respect to compactly supported perturbations. The latter property then implies that this solution also satisfies the asymptotic behaviour (\ref{eqY+}) (see \cite[Sec. 4]{clerc2017theory}).

It thus remains to show that \begin{equation}\label{eqastarInf}
                               \inf_{\left\{(\xi,\mu,f) \ \textrm{as in Thm. \ref{thm1}}\right\}}\alpha_*(\xi,\mu,f)\leq -\frac{1}{2}.
                             \end{equation}
To this end, as was observed in \cite{clerc2017theory}, the special choice $f=-\frac{\mu'}{2}$ yields that $a_*=\sqrt{2}$.
  So, this choice gives  $\alpha_*=-\frac{1}{2}$ which   suffices for the proof of the corollary.
\end{proof}

\bibliographystyle{plain}
\bibliography{biblioaacs}
\end{document}